\pdfoutput=1
\RequirePackage{fix-cm}
\documentclass{llncs}
\usepackage{microtype}
\usepackage[utf8]{inputenc}
\usepackage[T1]{fontenc}
\usepackage{mathtools,amssymb}
\usepackage{enumitem,booktabs}
\mathtoolsset{centercolon,mathic}
\usepackage{todonotes} 
\usepackage{import} 
\usepackage{xparse} 

\newif\ifarxiv
\arxivfalse

\ifarxiv
\usepackage{tikzexternal}
\else
\usepackage{tikz-cd}
\usetikzlibrary{
  calc ,
  matrix ,
  arrows ,
  intersections ,
  positioning ,
  fit ,
  shapes.geometric ,
  shapes.misc ,
  decorations ,
  decorations.pathmorphing ,
  decorations.pathreplacing ,
  decorations.text
}
\tikzset
{
	diagram/.style =
	{
		arrows = -> ,
		> = angle 60 ,
		auto = left ,
		text depth = 0.25ex , 
		font = \small			
	} ,
	code node/.append style =
	{
		every node/.append style =
		{
			execute at begin node = \begin{texttt} ,
			execute at end node = \end{texttt}
		}
	} ,
	code diagram/.style =
	{
		diagram ,
		code node
	} ,
	math node/.append style =
	{
		every node/.append style =
		{
			execute at begin node = \begin{math} ,
			execute at end node = \end{math}
		}
	} ,
	math diagram/.style =
	{
		diagram ,
		math node
	} ,
	string diagram/.style =
	{
		math node ,
	} ,
	online/.style =
	{
		shape = rectangle ,
		rounded corners = 2mm ,
		inner sep = 1.5pt ,
		fill = white ,
		opacity = 1.0 ,
		anchor = center	
	} ,
	double arrow/.style =
	{
		double distance between line centers = 2pt ,
		-implies 
	} ,
	mapto/.append style =
	{
		arrows = |-> ,
	} ,
	monic/.append style =
	{
		arrows = >-> ,
	} ,
	epic/.append style =
	{
		arrows = ->> ,
	} ,
	string/.append style =
	{
		arrows = - ,
	} ,
	overcross/.append style =
	{
		preaction = {draw = white , - , line width = 6pt}
	} ,
	label/.append style =
	{
		font = \scriptsize
	} ,
	forall/.append style =
	{
		line width = 0.8pt ,
	} ,
	exists/.append style =
	{
		densely dashed
	} ,
	2-cell/.style =
	{
		shape = rectangle ,
		rounded corners = 2mm ,
		minimum size = 4mm ,
		inner sep = 1.5pt ,
		draw = black!75 ,
		fill = white ,
		opacity = 0.9
	} ,
	1-cell/.style =
	{
		shape = rectangle ,
		rounded corners = 2mm ,
		inner sep = 1.5pt ,
		fill = white ,
		opacity = 1.0 ,
		anchor = center	
	} ,
	0-cell/.style =
	{
		shape = rectangle ,
		rounded corners = 2mm ,
		inner sep = 2pt
	} ,
	diamond/.style =
	{
		shape = diamond ,
		minimum size = 4mm ,
		inner sep = 1.0pt ,
		draw = black!75 ,
		fill = white ,
		opacity = 0.9
	} ,
}

\fi

\ifarxiv
\usepackage{url}
\usepackage[longnamesfirst,sort,sectionbib,round]{natbib}
\let\textcite\citet

\let\parencite\citep
\else
\usepackage[style=authoryear-comp,maxnames=10,backend=bibtex]{biblatex}
\fi
\usepackage[hidelinks,bookmarksdepth=2]{hyperref}

\makeatletter
\protected\def\tikz@nonactivecolon{\ifmmode\mathrel{\mathop\ordinarycolon}\else:\fi}
\makeatother

\ifarxiv
\relax
\else
\DeclareFieldFormat{eprint:euclid}{%
  Project Euclid\addcolon\space
  \ifhyperref
    {\href{http://projecteuclid.org/euclid.#1}{\nolinkurl{#1}}}
    {\nolinkurl{#1}}}
\DeclareFieldFormat{eprint:eudml}{%
  \mkbibacro{EUDML}\addcolon\space
  \ifhyperref
    {\href{https://eudml.org/doc/#1}{\nolinkurl{#1}}}
    {\nolinkurl{#1}}}
\DeclareFieldFormat{eprint:hal}{%
  \mkbibacro{HAL}\addcolon\space
  \ifhyperref
    {\href{https://hal.archives-ouvertes.fr/hal-#1/document}{\nolinkurl{#1}}}
    {\nolinkurl{#1}}}
\DeclareFieldFormat{eprint:jstor}{%
  \mkbibacro{JSTOR}\addcolon\space
  \ifhyperref
    {\href{http://www.jstor.org/stable/#1}{\nolinkurl{#1}}}
    {\nolinkurl{#1}}}
\DeclareFieldFormat{eprint:pmc}{%
  \mkbibacro{PMCID}\addcolon\space
  \ifhyperref
    {\href{http://www.ncbi.nlm.nih.gov/pmc/articles/#1/}{\nolinkurl{#1}}}
    {\nolinkurl{#1}}}
  
\fi 

\newcommand*{\C}{\mathbb{C}}
\newcommand*{\cst}[1]{\mathrm{#1}}
\newcommand*{\stdcat}[1]{\textup{\bfseries #1}}
\newcommand*{\Set}{\stdcat{Set}}
\newcommand*{\sSet}{\stdcat{sSet}}

\newcommand*{\Top}{\stdcat{Top}}
\newcommand*{\Hot}{\stdcat{Hot}}
\newcommand*{\Cat}{\stdcat{Cat}}

\newcommand*{\op}{^{\mathrm{op}}}
\newcommand*{\inv}{^{-1}}
\DeclareMathOperator{\Id}{Id}
\DeclareMathOperator{\Hom}{Hom}

\DeclareMathOperator{\yo}{y}
\DeclarePairedDelimiter\abs{\lvert}{\rvert}

\DeclareSymbolFont{bbold}{U}{bbold}{m}{n}
\DeclareSymbolFontAlphabet{\mathbbold}{bbold}
\newcommand*{\bbtwo}{\mathbbold{2}}
\newcommand*{\bbthree}{\mathbbold{3}}
\newcommand*{\zstr}{{\cdot}} 

\NewDocumentCommand \id {m}
{
	\Id
	\ifthenelse
	{\equal{#1}{}}
	{}
	{_{#1}} 
}
\RenewDocumentCommand \hom {o m m}
{%
	\IfNoValueTF{#1}%
	{{#2} \to {#3}}%
	{\Hom _{#1} \, ({#2} , {#3})}%
}
\NewDocumentCommand \define {s o m}
{%
	\hypertarget%
	{\IfValueTF{#2}{#2}{#3}}%
	{\IfBooleanTF{#1}{#3}{\emph{#3}}}%
}
\NewDocumentCommand \refer {o m}
{%
	\hyperlink%
	{\IfValueTF {#1}{#1}{#2}}%
	{#2}%
}
\NewDocumentCommand \degen {} {\ast}

\title{Varieties of Cubical Sets}
\date{\today}
\author{Ulrik Buchholtz\inst{1} \and Edward Morehouse\inst{2}}
\institute{Department of Philosophy, Carnegie Mellon University,\\
  Pittsburgh, PA 15213, USA\\
  \emph{Current address:}
  Fachbereich Mathematik, Technische Universit{\"a}t Darmstadt,\\
  Schlossgartenstra\ss e~7, 64289 Darmstadt, Germany\\
  \email{buchholtz@mathematik.tu-darmstadt.de}
\and
  Carnegie Mellon School of Computer Science,\\
  5000 Forbes Avenue, Pittsburgh, PA 15213, USA\\
  \emph{Current address:}
  Department of Mathematics and Computer Science,\\
  Wesleyan University,\\
  265 Church Street, Middletown, CT 06459, USA\\
  \email{emorehouse@wesleyan.edu}}

\begin{document}

\maketitle

\begin{abstract}
  We define a variety of notions of cubical sets, based on sites
  organized using substructural algebraic theories presenting PRO(P)s
  or Lawvere theories. We prove that all our sites are test categories
  in the sense of Grothendieck, meaning that the corresponding
  presheaf categories of cubical sets model classical homotopy
  theory. We delineate exactly which ones are even strict test
  categories, meaning that products of cubical sets correspond to
  products of homotopy types.
\end{abstract}

\section{Introduction}
\label{sec:introduction}
There has been substantial interest recently in the use of
cubical structure as a basis for higher-dimensional type theories
\parencite{Awodey2016,Bezem-Coquand-Huber2014,Cohen-Coquand-Huber-Mortberg2016,%
  Angiuli-Harper-Wilson2017}.
A question that quickly arises is, what sort of cubical structure,
because there are several plausible candidates to choose from.

The suitability of a given cubical structure as a basis for a higher-dimensional type theory
is dependent on at least two considerations:
its proof-theoretic characteristics (e.g. completeness with respect to a given class of models,
decidability of equality of terms, existence of canonical forms)
and its homotopy-theoretic characteristics, most importantly,
that the synthetic homotopy theory to which it gives rise should agree
with the standard homotopy theory for topological spaces.

\paragraph{Contributions.} In this paper we organize various notions of cubical sets along several axes, using substructural algebraic theories as a guiding principle~\parencite{Mauri2005}.
We define a range of \emph{cube categories} (or \emph{cubical sites}), categories $\C$ for which the corresponding presheaf category $\widehat\C$ can be thought of as a category of cubical sets.
We then consider each of these from the perspective of \emph{test categories} \parencite{Grothendieck1983}, which relates presheaf categories and homotopy theory.
We give a full analysis of our cubical sites as test categories.

\subsection{Cube Categories}

Our cube categories are presented
as monoidal categories with a single generating object, $X$, representing an abstract dimension.
To give such a presentation, then, is to give a collection of morphism generators, $f : X^{\otimes n} \to X^{\otimes m}$,
and a collection of equations between parallel morphisms.
Such a presentation is known in the literature as a ``PRO''.
Two closely-related notions are those of ``PROP'' and of ``Lawvere theory'',
the difference being that a PROP assumes that the monoidal category is \emph{symmetric},
and a Lawvere theory further assumes that it is \emph{cartesian}.
These additional assumptions manifest themselves in the proof theory as the admissibility of certain structural rules:
in a PROP, the structural rule of \emph{exchange} is admissible,
and in a Lawvere theory \emph{weakening} and \emph{contraction} are admissible as well.

The distinguishing property of a theory making it ``cubical'', rather than of some other ``shape'',
is the presence of two generating morphisms, $d^0 , d^1 : 1 \to X$ representing the face maps (where ``$1$'' is the monoidal unit).
This is because combinatorially, an $n$-dimensional cube has $2 n$ many $(n - 1)$-cube faces, namely, two in each dimension.

\subsection{Test Categories}
\label{sec:testcategories}

In his epistolary research diary, \emph{Pursuing Stacks}, Grothendieck set out a program for the study of abstract homotopy.
The homotopy theory of topological spaces can be described as that of (weak) higher-dimensional groupoids.
Higher-dimensional groupoids also arise from combinatorially-presented structures, such as simplicial or, indeed, cubical sets, which can thus be seen as alternative presentations of the classical homotopy category, $\Hot$. That is, $\Hot$ arises either from the category of topological spaces, or from the category of simplicial sets, by inverting a class of morphisms (the weak equivalences).

Recall that for any small category $A$
and any cocomplete category $\mathcal E$, there is an equivalence between functors $i : A \to \mathcal E$ and adjunctions
\begin{equation}\label{eq:adjunction}
  \begin{tikzpicture}[commutative diagrams/every diagram]
    \matrix[matrix of math nodes, name=m, commutative diagrams/every cell] {
      \widehat A & \mathcal E \\};
    \path[commutative diagrams/.cd,every arrow,every label,shift right=0.4em]
      (m-1-1) to node[swap] {$i_!$} (m-1-2);
    \path[commutative diagrams/.cd,every arrow,every label,shift right=0.4em]
      (m-1-2) to node[swap] {$i^*$} (m-1-1);
    \path[commutative diagrams/.cd,every arrow,every label,phantom]
      (m-1-2) to node[commutative diagrams/description,scale=0.5] {$\top$} (m-1-1);
  \end{tikzpicture}
\end{equation}
where $i_!$ is the left Kan extension of $i$, and $i^*(X)(a) = \Hom_{\mathcal E}(i(a),X)$.\footnote{%
In this paper, all small categories are considered as strict categories, and $\Cat$ denotes the $1$-category of these.}
Here $\widehat A$ is the category of functors $A\op \to \Set$, also known as the presheaf topos of $A$.
The category of simplicial sets is defined as the presheaf category $\sSet=\widehat \Delta$, where $\Delta$ is the category of non-empty finite totally ordered sets and order-preserving maps.

The category of small categories, $\Cat$, likewise presents the homotopy category, by inverting the functors that become weak equivalences of simplicial sets after applying the nerve functor, which is the right adjoint functor $N=i^*$ coming from the inclusion $\Delta\hookrightarrow\Cat$.
Grothendieck realized that this allows us to compare the presheaf category $\widehat A$, for any small category $A$, with the homotopy category in a canonical way, via the functor $i_A : A \to \Cat$ given by $i_A(a) = A_{/a}$.
The left Kan extension is in this case also denoted $i_A$ and we have $i_A(X) = A_{/X}$ for $X$ in $\widehat A$.
He identified the notion of a \emph{test category}, which is one for which the homotopy category of its category of presheaves is equivalent to the homotopy category via the right adjoint $i_A^*$, after localization at the weak equivalences.
In other words, presheaves over a test category are models for homotopy types of spaces/$\infty$-groupoids.
In this way, the simplex category $\Delta$ is a test category, and the classical cube category of Serre and Kan, $\C_{(\cst w,\zstr)}$ in our notation, is also a test category.

One perceived benefit of the simplex category $\Delta$ is that the induced functor $\sSet \to \Hot$ preserves products, whereas the functor $\widehat{\C_{(\cst w,\zstr)}}\to\Hot$ does not.
A test category $A$ for which the functor $\widehat A\to\Hot$ preserves products is called a \emph{strict} test category.
We shall show that most natural cube categories are in fact strict test categories.

\section{Cube categories}
\label{sec:cubecategories}

In this paper, we consider as base categories the syntactic categories
for a range of monoidal theories, all capturing some aspect
of the notion of an interval. We call the resulting categories ``cube
categories'', because the monoidal powers of the interval then
correspond to cubes\footnote{Here we say ``cube'' for short instead
  of ``hypercube'' for an arbitrary dimensional power of an
  interval. A prefix will indicate the dimension, as in $0$-cube,
  $1$-cube, etc.}.

The generating morphisms for these cube categories can be classified either as \emph{structural},
natural families corresponding to structural rules of a proof theory,
or as \emph{algebraic}, distinguished by the property of having coarity one.
Our cube categories vary along three dimensions: the structural rules
present, the signature of algebraic function symbols, and the
equational theory.

This section is organized as follows: in subsection~\ref{sec:monoidal} we recall the basics of algebraic theories in monoidal categories, and in subsection~\ref{sec:interpretations} we discuss how monoidal languages are interpreted in monoidal categories.
Then we introduce in subsection~\ref{sec:languages} the monoidal languages underlying our cubical theories, so that in subsection~\ref{sec:canonicalcubes} we can introduce the standard interpretations of these and our cubical theories. In subsection~\ref{sec:menagerie} we give a tour of the resulting cube categories.

\subsection{Monoidal algebraic theories}
\label{sec:monoidal}

We assume the reader is familiar with ordinary algebraic theories and their categorical incarnations as Lawvere theories.
The idea of monoidal algebraic theories as described by \textcite{Mauri2005} is to generalize this to cases where only a subset of the structural rules (weakening, exchange and contraction) are needed to describe the axioms.
Think for example of the theory of monoids over a signature of a neutral element $1$ and a binary operation.
The axioms state that $1x=x=x1$ (in the context of one variable $x$) and $x(yz)=(xy)z$ (in the context of the variables $x$, $y$ and $z$).
Note that the terms in these equations each contain all the variables of the context and in the same order.
Thus, no structural rules are needed to form the equations, and hence the notion of a monoid makes sense in any monoidal category.

Let us now make this more precise.
For the structural rules, we follow Mauri (\emph{ibid.}) and consider
any subset of $\{\cst w,\cst e,\cst c\}$ ($\cst w$ for weakening,
$\cst e$ for exchange, and $\cst c$ for contraction), except that
whenever $\cst c$ is present, so is $\cst e$. Thus we consider the
following lattice of subsets of structural rules:
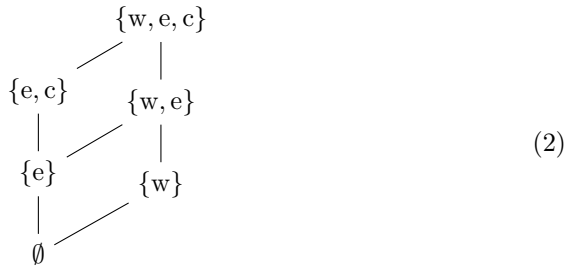
\begin{equation}\label{eq:lattice-str-rules}
  \begin{tikzpicture}[baseline=(e.north),node distance=5mm]
    \node(wec) {$\{\cst w,\cst e,\cst c\}$};
    \node(we) [below=of wec] {$\{\cst w,\cst e\}$};
    \node(w) [below=of we] {$\{\cst w\}$};
    \node(ec) [below left=of wec] {$\{\cst e,\cst c\}$};
    \node(e) [below=of ec] {$\{\cst e\}$};
    \node(p) [below=of e] {$\emptyset$};
    \draw (p)--(e)--(ec)--(wec);
    \draw (p)--(w)--(we)--(wec);
    \draw (e)--(we);
  \end{tikzpicture}
\end{equation}
It would be possible to consider monoidal theories over operations that have any number of incoming and outgoing edges, representing morphisms $f : A_1 \otimes \dots \otimes A_n \to B_1 \otimes \dots \otimes B_m$ in a monoidal category,
where the $A_i$ and $B_j$ are sorts.
However, it simplifies matters considerably, and suffices for our purposes, to consider only operations of the usual kind in algebra, with any arity of incoming edges and exactly one outgoing edge, as in $f: A_1 \otimes \dots \otimes A_n \to B$.%

Furthermore, we shall consider only single-sorted theories, defined over a signature $\Sigma$ given by a set of function symbols with arities.
For a single-sorted signature, the types can of course be identified with the natural numbers.

A (single-sorted, algebraic) \emph{monoidal language} consists of a pair of a set of structural rules together with an algebraic signature.
A term $t$ in a context of $n$ free variables $x_1,\dots,x_n$ is built up from the function symbols in such a way that
when we list the free variables in $t$ from left to right:
\begin{itemize}
\item every variables occurs unless $\cst w$ (weakening) is a structural rule;
\item the variables occur in order unless $\cst e$ (exchange) is a structural rule; and
\item there are no duplicated variables unless $\cst c$ (contraction) is a structural rule.
\end{itemize}
For the precise rules regarding term formation and the proof theory of equalities of terms, we refer to \textcite{Mauri2005}.

\subsection{Interpretations}
\label{sec:interpretations}

An interpretation of a monoidal language in a monoidal category
$(\mathcal E,1,\otimes)$ will consist of an object $X$ representing the single
sort, together with morphisms representing the structural rules and
morphisms representing the function symbols (including constants).

The structural rules $\cst w$, $\cst e$, $\cst c$ are interpreted
respectively by morphisms
\begin{align*}
  \varepsilon &: X \to 1 \\
  \tau &: X\otimes X \to X\otimes X \\
  \delta &: X \to X \otimes X
\end{align*}
satisfying certain laws \parencite[(27--36)]{Mauri2005}.
These laws specify also the interaction between these morphisms and the morphisms corresponding to the function symbols.
When $\mathcal E$ is symmetric monoidal, we
interpret $\tau$ by the braiding of $\mathcal E$, and when $\mathcal
E$ is cartesian monoidal, we interpret everything using the cartesian
structure.
A function symbol $f$ of arity $n$ is interpreted by a morphism $\abs f : X^{\otimes n} \to X$. This morphism and the structural morphisms are required to interact nicely, e.g., $\varepsilon\circ\abs f = \varepsilon^{\otimes n}$ (cf.~\emph{loc.~cit.}).

In a syntactic category for the empty theory (of which the syntactic
category for a non-empty theory is a quotient), every morphism factors
uniquely as a structural morphism followed by a functional one
(\emph{op.\ cit.}, Prop.~5.1).
When we impose a theory, we may of
course lose uniqueness, but we still have existence.


\subsection{Cubical monoidal languages}
\label{sec:languages}

All of our cubical signatures will include the two endpoints of the interval
as nullary function symbols, $0$ and $1$. For the rest, we consider
the two ``connections'' $\vee$ and $\wedge$, as well as the reversal,
indicated by a prime, ${}'$. This gives us the following lattice of
6 signatures:
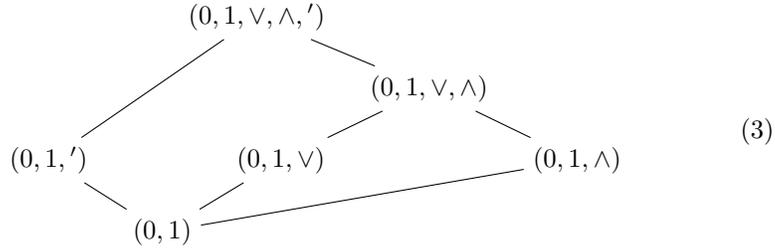
\begin{equation}\label{eq:lattice-sig}
  \begin{tikzpicture}[baseline=(m.north),node distance=5mm]
    \node(mjr) {$(0,1,\vee,\wedge,{}')$};
    \node(mj) [below right=of mjr] {$(0,1,\vee,\wedge)$};
    \node(j) [below left=of mj] {$(0,1,\vee)$};
    \node(m) [below right=of mj] {$(0,1,\wedge)$};
    \node(b) [below left=of j] {$(0,1)$};
    \node(r) [above left=of b] {$(0,1,{}')$};
    \draw (b)--(r)--(mjr);
    \draw (b)--(j)--(mj)--(mjr);
    \draw (b)--(m)--(mj);
  \end{tikzpicture}
\end{equation}
Combined with the 6 possible combinations of structural rules, we are thus dealing with 36 distinct languages.

\begin{definition}\label{def:language}
  Let $L_{(a,b)}$, where $a$ is one of the six substrings of ``$\cst{wec}$'' corresponding to~\eqref{eq:lattice-str-rules} and $b$
  one of the six substrings of ``${\vee}{\wedge}{}'$'' corresponding to~\eqref{eq:lattice-sig}, denote the language with
  structural rules from $a$ and signature obtained from $(0,1)$ by
  expansion with the elements of $b$.
\end{definition}

We shall, however, mostly consider the 18 languages with weakening present (otherwise we are dealing with \emph{semi-cubical sets}).

The third dimension of variation for our cube categories is the
algebraic theory for a given language. Here we shall mostly pick the
theory of a particular standard structure, so let us pause our
discussion of cube categories to introduce our standard structures.

\subsection{The canonical cube categories}
\label{sec:canonicalcubes}

We consider the interval structures of the following objects in
the cartesian monoidal categories \Top\ and \Set\ (so
all structural rules are supported, interpreted by the global structure):
\begin{itemize}
\item the standard topological interval, $I=[0,1]$ in \Top;
\item the standard $2$-element set, $I=\bbtwo=\{0,1\}$ in \Set.
\end{itemize}
For the topological interval, we take $x\vee y=\max\{x,y\}$, $x\wedge
y=\min\{x,y\}$ and $x'=1-x$. These formulas also apply to the
$2$-element set $\bbtwo$.

We can also consider the $3$-element Kleene algebra (cf.~subsection~\ref{sec:menagerie})
$\bbthree=\{0,u,1\}$ with $u'=u$ and the $4$-element de~Morgan
algebra $\mathbb D=\{0,u,v,1\}$ with $u'=u$ and $v'=v$ in \Set\
(called the \emph{diamond}), as further structures. It is interesting
to note here that $\bbthree$ has the same theory as $[0,1]$ in the
full language, namely the theory of Kleene
algebras \parencite{Gehrke-et-al2003}. Of course, $\bbtwo$ gives us
the theory of Boolean algebras, while $\mathbb D$ gives the theory of
de~Morgan algebras.

\begin{definition}\label{def:cubecat}
  With $(a,b)$ as in Def.~\ref{def:language}, and $T$ a theory
  in the language $L_{(a,b)}$, let $\C_{(a,b)}(T)$ denote the
  syntactic category of the theory $T$. We write $\C_{(a,b)}$ for
  short for $\C_{(a,b)}(\mathrm{Th}([0,1]))$ for the syntactic
  category of the topological interval with respect to the monoidal
  language $L_{(a,b)}$.
\end{definition}
We say that $\C_{(a,b)}$ is the \emph{canonical} cube category for the
language $L_{(a,b)}$.
\begin{proposition}
  The canonical cube category $\C_{(a,b)}$ is isomorphic to the
  monoidal subcategory of $\Top$ generated by $[0,1]$ with respect to
  the language $L_{(a,b)}$.
\end{proposition}
In fact, since the forgetful functor $\Top\to\Set$ is faithful, we get
the same theory in every language whether we consider the interval
$[0,1]$ in $\Top$ or in $\Set$, and it generates the same monoidal
subcategory in either case.
\begin{proposition}
  If $L$ is a cubical language strictly smaller than the maximal
  language, $L_{(\cst{wec},\vee\wedge{}')}$, then the theory of
  the standard structure $[0,1]$ is the same as the theory of the
  standard structure $\bbtwo=\{0,1\}$ in $\Set$. For
  $L_{(\cst{wec},\vee\wedge{}')}$ the theory of $[0,1]$ equals the
  theory of the three-element Kleene algebra, $\bbthree$, and is in
  fact the theory of Kleene algebras.
\end{proposition}
\begin{proof}
  If $L$ does not have reversal, then we can find for any points $a<b$ in $[0,1]$ a homomorphism $f : [0,1]\to\bbtwo$ in \Set\ with $f(a)<f(b)$ (cut between $a$ and $b$), so $[0,1]$ and $\bbtwo$ have the same theory.
  
  If $L$ does not have contraction, each term is monotone in each variable.
  By pushing reversals towards the variables and applying the absorption laws, we can find for each resulting term $s[x,y_1,\dots y_n]$ instantiations $b_1,\dots,b_n\in\bbtwo$ with $s[0,b_1,\dots,b_n] \ne s[1,b_1,\dots,b_n]$.
  It follows that if two terms agree on instantiations in $\bbtwo$, then they use each variable the same way: either with or without a reversal.
  Hence we can ignore the reversals and use the previous case.
  
  The last assertion is proved by \textcite{Gehrke-et-al2003}.\qed
\end{proof}
\begin{corollary}
  For each canonical cube category $\C_{(a,b)}$ we have decidable equality of terms.
\end{corollary}
\begin{proof}
  This follows because equality is decidable in the theory of Kleene algebras, as this theory is characterized by the object $\bbthree$.
  Hence we can decide equality by the method of ``truth-tables'', however, relative to the elements of $\bbthree$.\qed
\end{proof}
Of course, any algorithm for equality in the theory of Kleene algebras can be used, including the one based on disjunctive normal forms (\emph{op.~cit.}).

We can in the presence of weakening quite simply give explicit axiomatizations for each of the
canonical cube categories $\C_{(a,b)}$: take the subset of axioms in
Tab.~\ref{tab:axioms} that make sense in the language $L_{(a,b)}$.
Without weakening, this will not suffice, as we should need infinitely many axioms to ensure for instance that $s \wedge 0 = t \wedge 0$ where $s,t$ are terms in the same variable context.
\begin{table}[t]
  \centering
  \setlength{\tabcolsep}{6pt}
  \begin{tabular}{ccl} \toprule
    Axiom & Lang.\ req.{} & Name \\ \midrule
    $x \vee (y \vee z) = (x \vee y) \vee z$ & $(\zstr,\vee)$ & $\vee$-associativity \\
    $0\vee x=x=x\vee 0$ & $(\zstr,\vee)$ & $\vee$-unit \\
    $1\vee x=1=x\vee 1$ & $(\cst w,\vee)$ & $\vee$-absorption \\
    $x\vee y=y\vee x$ & $(\cst e,\vee)$ & $\vee$-symmetry \\
    $x\vee x=x$ & $(\cst{ec},\vee)$ & $\vee$-idempotence \\ \midrule
    $x \wedge (y \wedge z) = (x \wedge y) \wedge z$ & $(\zstr,\wedge)$ & $\wedge$-associativity \\
    $1\wedge x=x=x\wedge 1$ & $(\zstr,\wedge)$ & $\wedge$-unit \\
    $0\wedge x=0=x\wedge 0$ & $(\cst w,\wedge)$ & $\wedge$-absorption \\
    $x\wedge y=y\wedge x$ & $(\cst e,\wedge)$ & $\wedge$-symmetry \\
    $x\wedge x=x$ & $(\cst{ec},\wedge)$ & $\wedge$-idempotence \\ \midrule
    $x''=x$ & $(\zstr,{}')$ & ${}'$-involution \\
    $0'=1$ & $(\zstr,{}')$ & ${}'$-computation \\ \midrule
    $x\wedge(y\vee z)=(x\wedge y)\vee(x\wedge z)$ & $(\cst{ec},\vee\wedge)$ 
                                                  & distributive law 1 \\
    $x\vee(y\wedge z)=(x\vee y)\wedge(x\vee z)$ & $(\cst{ec},\vee\wedge)$ 
                                                  & distributive law 2 \\
    $x=x\vee(x\wedge y)=x\wedge(x\vee y)$ & $(\cst{wec},\vee\wedge)$
                                          & lattice-absorption \\ 
    \midrule
    $(x\vee y)'=x'\wedge y'$ & $(\zstr,\vee\wedge')$ & de~Morgan's law \\
    $x\wedge x'\le y\vee y'$ & $(\cst{wec},\vee\wedge')$ & Kleene's law \\
    \bottomrule
  \end{tabular}
  \vspace{6pt}

  \caption{Cubical axioms}
  \label{tab:axioms}
\end{table}

When we compare our canonical cube categories to the categories of \textcite{Grandis-Mauri2003}, we find
$\mathbb I=\C_{(\cst w,\zstr)}$, $\mathbb J=\C_{(\cst w,\vee\wedge)}$,
$\mathbb K=\C_{(\cst{we},\vee\wedge)}$, and
$!\mathbb K=\C_{(\cst{we},\vee\wedge')}$.

For the full language, $L_{(\cst{wec},{\vee}{\wedge}')}$, there are two additional interesting cube categories: $\C_{\mathrm{dM}}$ (de~Morgan algebras) and $\C_{\mathrm{BA}}$ (boolean algebras), where de~Morgan algebras satisfy all the laws of Tab.~\ref{tab:axioms} except Kleene's law, and boolean algebras of course satisfy additionally the law $x\vee x'=1$. The case of de~Morgan algebras is noteworthy as the basis of the cubical model of type theory of \textcite{Cohen-Coquand-Huber-Mortberg2016}.

\begin{definition}
  For each of our cube categories $\C$, we write $[n]$ for the
  object representing a context of $n$ variables, and we write
  $\square^n$ for the image of $[n]$ under the Yoneda embedding
  $\yo: \C\hookrightarrow\widehat\C$.
\end{definition}

\subsection{A tour of the menagerie}
\label{sec:menagerie}

\subsubsection{Plain cubes} \label{sec:plain_cubes}

In the canonical cube category for the language $L_{(\cst{w},\zstr)}$ the monoidal unit $[0]$ is terminal
and the interpretation of weakening is generated by the unique degeneracy map $\varepsilon : \hom{[1]}{[0]}$.
The points $0$ and $1$ are interpreted by face maps $\eta_0 , \eta_1 : \hom{[0]}{[1]}$.
In the following we let $\{i,j\}$ range over $\{0,1\}$ with the assumption that $i \not= j$.

Because $[0]$ is terminal we have the \define[face-degeneracy law]{face-degeneracy laws},
$\varepsilon \circ \eta_i = \id{[0]}$,
which we represent string-diagrammatically as:
\[
\begin{tikzpicture}[string diagram , baseline=(current bounding box.center)]
	\coordinate (tl) at (-0.6 , 0.6) ;
	\coordinate (br) at (0.6 , -0.6) ;
	\path [draw , clip] (tl) rectangle (br) ;
	\coordinate (eta) at ({$ (tl) ! 1/2 ! (br) $} |- {$ (tl) ! 1/4 ! (br) $}) ;
	\coordinate (eps) at ({$ (tl) ! 1/2 ! (br) $} |- {$ (tl) ! 3/4 ! (br) $}) ;
	\draw (eta) to [out = south , in = north] (eps) ;
	\node [2-cell] at (eta) {\eta_i} ;
	\node [2-cell] at (eps) {\varepsilon} ;
\end{tikzpicture}
\quad = \quad
\begin{tikzpicture}[string diagram , baseline=(current bounding box.center)]
	\coordinate (tl) at (-0.6 , 0.6) ;
	\coordinate (br) at (0.6 , -0.6) ;
	\path [draw , clip] (tl) rectangle (br) ;
\end{tikzpicture}
\]
This law is invisible in the algebraic notation.

In a presheaf $T$, the face maps give rise to reindexing functions between the fibers,
$\partial ^i _k : \hom{T [n+1]}{T [n]}$, where $1 \le k \le n+1$.
In particular, when $n$ is $0$ these pick out the respective boundary points of an interval.
Similarly, the degeneracy map gives rise to reindexing functions between the fibers,
$\degen_k : \hom{T [n]}{T [n + 1]}$, where $1 \le k \le n+1$.
In particular, when $n$ is $0$ this determines a degenerate interval $\degen(a)\in T[1]$ on a point $a\in T[0]$,
and the \refer[face-degeneracy law]{face-degeneracy laws} tell us its boundary points: $\partial^0(\degen(a))=\partial^1(\degen(a))=a$.

Adding the structural law of exchange to the language adds a natural isomorphism $\tau : \hom{[2]}{[2]}$ to the syntactic category.
In a presheaf, this lets us permute any adjacent pair of a cube's dimensions by reflecting it across the corresponding diagonal (hyper)plane.
In particular, $T (\tau)$ reflects a square across its main diagonal.

\subsubsection{Cubes with diagonals} \label{sec:diagonal_cubes}

In the canonical cube category for the language $L_{(\cst{wec},\zstr)}$ the monoidal product is cartesian
and the interpretation of contraction is generated by the diagonal map $\delta : \hom{[1]}{[2]}$.
In this case, the pair $(\delta , \varepsilon)$ forms a cocommutative comonoid.

Because $\delta$ is natural we have the \define[face-diagonal law]{face-diagonal laws},
$\delta \circ \eta_i = \eta_i \otimes \eta_i$,
which we represent string-diagrammatically as:
\[
\begin{tikzpicture}[string diagram , baseline=(current bounding box.center)]
	\coordinate (tl) at (-0.6 , 0.6) ;
	\coordinate (br) at (0.6 , -0.6) ;
	\path [draw , clip] (tl) rectangle (br) ;
	\coordinate (eta) at ({$ (tl) ! 1/2 ! (br) $} |- {$ (tl) ! 5/20 ! (br) $}) ;
	\coordinate (del) at ({$ (tl) ! 1/2 ! (br) $} |- {$ (tl) ! 14/20 ! (br) $}) ;
	\coordinate (out1) at ({$ (tl) ! 1/6 ! (br) $} |- br) ;
	\coordinate (out2) at ({$ (tl) ! 5/6 ! (br) $} |- br) ;
	\draw (eta) to [out = south , in = north] (del) ;
	\draw (del) to [out = west , in = north] (out1) ;
	\draw (del) to [out = east , in = north] (out2) ;
	\node [2-cell] at (eta) {\eta_i} ;
	\node [2-cell] at (del) {\delta} ;
\end{tikzpicture}
\quad = \quad
\begin{tikzpicture}[string diagram , baseline=(current bounding box.center)]
	\coordinate (tl) at (-0.6 , 0.6) ;
	\coordinate (br) at (0.6 , -0.6) ;
	\path [draw , clip] (tl) rectangle (br) ;
	\coordinate (eta1) at ({$ (tl) ! 1/4 ! (br) $} |- {$ (tl) ! 1/2 ! (br) $}) ;
	\coordinate (eta2) at ({$ (tl) ! 3/4 ! (br) $} |- {$ (tl) ! 1/2 ! (br) $}) ;
	\coordinate (out1) at (eta1 |- br) ;
	\coordinate (out2) at (eta2 |- br) ;
	\draw (eta1) to [out = south , in = north] (out1) ;
	\draw (eta2) to [out = south , in = north] (out2) ;
	\node [2-cell] at (eta1) {\eta_i} ;
	\node [2-cell] at (eta2) {\eta_i} ;
\end{tikzpicture}
\]

In a presheaf $T$, the diagonal map gives rise to reindexing functions between the fibers,
$d _k : \hom{T [n+2]}{T [n+1]}$, where $1 \le k \le n+1$.
In particular, when $n$ is $0$ this picks out the main diagonal of a square,
and the \refer[face-diagonal law]{face-diagonal laws} tell us its boundary points.

In the cubical semantics, the fact that $\varepsilon$ is a comonoid counit for $\delta$ tells us
that the diagonal of a square formed by degenerating an interval is just that interval.
Likewise, the fact that $\delta$ is coassociative tells us
that the main diagonal interval of a higher-dimensional cube is well-defined.

\subsubsection{Cubes with reversals} \label{sec:reversal_cubes}

In the canonical cube category for the language $L_{(\zstr, \cst{'})}$ we have an involutive reversal map $\rho : \hom{[1]}{[1]}$.
A reversal acts by swapping the face maps, giving the \define[face-reversal law]{face-reversal laws},
$\rho \circ \eta_i = \eta_j$,
which we represent string-diagrammatically as:
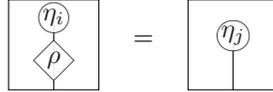
\[
\begin{tikzpicture}[string diagram , baseline=(current bounding box.center)]
	\coordinate (tl) at (-0.6 , 0.6) ;
	\coordinate (br) at (0.6 , -0.6) ;
	\path [draw , clip] (tl) rectangle (br) ;
	\coordinate (eta) at ({$ (tl) ! 1/2 ! (br) $} |- {$ (tl) ! 1/5 ! (br) $}) ;
	\coordinate (out) at (eta |- br) ;
	\draw (eta) to [out = south , in = north] coordinate [pos = 3/5] (rho) (out) ;
	\node [2-cell] at (eta) {\eta_i} ;
	\node [diamond] at (rho) {\rho} ;
\end{tikzpicture}
\quad = \quad
\begin{tikzpicture}[string diagram , baseline=(current bounding box.center)]
	\coordinate (tl) at (-0.6 , 0.6) ;
	\coordinate (br) at (0.6 , -0.6) ;
	\path [draw , clip] (tl) rectangle (br) ;
	\coordinate (eta) at ({$ (tl) ! 1/2 ! (br) $} |- {$ (tl) ! 2/5 ! (br) $}) ;
	\coordinate (out) at (eta |- br) ;
	\draw (eta) to [out = south , in = north] (out) ;
	\node [2-cell] at (eta) {\eta_j} ;
\end{tikzpicture}
\]
This embodies the equations $0'=1$ and $1'=0$.

In a presheaf $T$, the reversal map gives rise to endomorphisms on the fibers,
$! _k : \hom{T [n+1]}{T [n+1]}$, where $1 \le k \le n+1$.
In particular, when $n$ is $0$ this reverses an interval,
and the \refer[face-reversal law]{face-reversal laws} tell us its boundary points.

If the signature contains weakening, then we have the \define{reversal-degeneracy law},
$\varepsilon \circ \rho = \varepsilon$.
In the cubical interpretation, this says that a degenerate interval is invariant under reversal.
If the signature contains contraction, then we have the \define{reversal-diagonal law},
$\delta \circ \rho = (\rho \otimes \rho) \circ \delta$.
Cubically, this says that the reversal of a square's diagonal is the diagonal of the square resulting from reversing in each dimension.

\subsubsection{Cubes with connections} \label{sec:connected_cubes}

In the canonical cube category for the language $L_{(\cst{w},\vee\wedge)}$,
the connectives are interpreted by connection maps $\mu_0 , \mu_1 : \hom{[2]}{[1]}$.
Each pair $(\mu_i , \eta_i)$ forms a monoid.
Furthermore, the unit for each monoid is an absorbing element for the other:
$\mu_j \circ (\eta_i \otimes \id{}) = \eta_i \circ \varepsilon = \mu_j \circ (\id{} \otimes \eta_i)$.
\textcite{Grandis-Mauri2003} call such structures ``dioids'',
so we refer to these as the \define[dioid absorption law] {dioid absorption laws},
and represent them string-diagrammatically as:
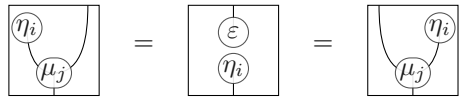
\[
\begin{tikzpicture}[string diagram , baseline=(current bounding box.center)]
	\coordinate (tl) at (-0.6 , 0.6) ;
	\coordinate (br) at (0.6 , -0.6) ;
	\path [draw , clip] (tl) rectangle (br) ;
	\coordinate (eta) at ({$ (tl) ! 1/5 ! (br) $} |- {$ (tl) ! 1/4 ! (br) $}) ;
	\coordinate (mu) at ({$ (tl) ! 1/2 ! (br) $} |- {$ (tl) ! 3/4 ! (br) $}) ;
	\coordinate (in) at ({$ (tl) ! 7/8 ! (br) $} |- tl) ;
	\coordinate (out) at (mu |- br) ;
	\draw (in) to [out = south , in = east] (mu) ;
	\draw (eta) to [out = south , in = west] (mu) ;
	\draw (mu) to [out = south , in = north] (out) ;
	\node [2-cell] at (eta) {\eta_i} ;
	\node [2-cell] at (mu) {\mu_j} ;
\end{tikzpicture}
\quad = \quad
\begin{tikzpicture}[string diagram , baseline=(current bounding box.center)]
	\coordinate (tl) at (-0.6 , 0.6) ;
	\coordinate (br) at (0.6 , -0.6) ;
	\path [draw , clip] (tl) rectangle (br) ;
	\coordinate (epsilon) at ({$ (tl) ! 1/2 ! (br) $} |- {$ (tl) ! 3/10 ! (br) $}) ;
	\coordinate (eta) at ({$ (tl) ! 1/2 ! (br) $} |- {$ (tl) ! 7/10 ! (br) $}) ;
	\coordinate (in) at (epsilon |- tl) ;
	\coordinate (out) at (eta |- br) ;
	\draw (in) to [out = south , in = north] (epsilon) ;
	\draw (eta) to [out = south , in = north] (out) ;
	\node [2-cell] at (epsilon) {\varepsilon} ;
	\node [2-cell] at (eta) {\eta_i} ;
\end{tikzpicture}
\quad = \quad
\begin{tikzpicture}[string diagram , baseline=(current bounding box.center)]
	\coordinate (tl) at (-0.6 , 0.6) ;
	\coordinate (br) at (0.6 , -0.6) ;
	\path [draw , clip] (tl) rectangle (br) ;
	\coordinate (eta) at ({$ (tl) ! 4/5 ! (br) $} |- {$ (tl) ! 1/4 ! (br) $}) ;
	\coordinate (mu) at ({$ (tl) ! 1/2 ! (br) $} |- {$ (tl) ! 3/4 ! (br) $}) ;
	\coordinate (in) at ({$ (tl) ! 1/8 ! (br) $} |- tl) ;
	\coordinate (out) at (mu |- br) ;
	\draw (in) to [out = south , in = west] (mu) ;
	\draw (eta) to [out = south , in = east] (mu) ;
	\draw (mu) to [out = south , in = north] (out) ;
	\node [2-cell] at (eta) {\eta_i} ;
	\node [2-cell] at (mu) {\mu_j} ;
\end{tikzpicture}
\]
Equationally, these amount to the ${\vee}$- and ${\wedge}$-absorption laws of Tab.~\ref{tab:axioms}.

In the cubical semantics, connections can be seen as a variant form of degeneracy,
identifying adjacent, rather than opposite, faces of a cube.
Of course in order for this to make sense, a cube must have at least two dimensions.
In a presheaf $T$, the connection maps give rise to reindexing functions between the fibers,
$\ulcorner _k , \lrcorner _k : \hom{T [n+1]}{T [n+2]}$, where $1 \le k \le n+1$.
In particular, when $n$ is $0$ these act as follows:
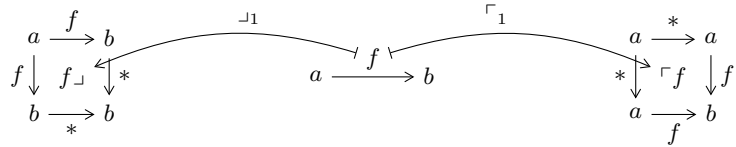
\[
\begin{tikzpicture}[math diagram , baseline=(current bounding box.center)]
	\coordinate (cl) at (-4.0 , 0.0) ;
	\coordinate (cc) at (0.0 , 0.0) ;
	\coordinate (cr) at (4.0 , 0.0) ;
	\node (ne0) at ($ (cl) + (-0.5 , 0.5) $) {a} ;
	\node (nw0) at ($ (cl) + (0.5 , 0.5) $) {b} ;
	\node (se0) at ($ (cl) + (-0.5 , -0.5) $) {b} ;
	\node (sw0) at ($ (cl) + (0.5 , -0.5) $) {b} ;
	\node (kappa0) at (cl) {f \lrcorner} ;
	\node (l) at ($ (cc) + (-0.75 , 0) $) {a} ;
	\node (r) at ($ (cc) + (0.75 , 0) $) {b} ;
	\node (ne1) at ($ (cr) + (-0.5 , 0.5) $) {a} ;
	\node (nw1) at ($ (cr) + (0.5 , 0.5) $) {a} ;
	\node (se1) at ($ (cr) + (-0.5 , -0.5) $) {a} ;
	\node (sw1) at ($ (cr) + (0.5 , -0.5) $) {b} ;
	\node (kappa1) at (cr) {\ulcorner f} ;
	\draw (l) to node [auto] (f) {f} (r) ;
	\draw (ne0) to node [auto] {f} (nw0) ;
	\draw (nw0) to node [auto] {\ast} (sw0) ;
	\draw (ne0) to node [swap] {f} (se0) ;
	\draw (se0) to node [swap] {\ast} (sw0) ;
	\draw (ne1) to node [auto] {\ast} (nw1) ;
	\draw (nw1) to node [auto] {f} (sw1) ;
	\draw (ne1) to node [swap] {\ast} (se1) ;
	\draw (se1) to node [swap] {f} (sw1) ;
	\draw [mapto] (f) to [bend right = 20] node [swap , pos = 2/5] {\lrcorner _1} (kappa0) ;
	\draw [mapto] (f) to [bend left = 20] node [auto , pos = 2/5] {\ulcorner _1} (kappa1) ;
\end{tikzpicture}
\]
It may help to think of the interval $f$ as a folded paper fan,
with its ``hinge'' at the domain end in the case of $\lrcorner$, and at the codomain end in the case of $\ulcorner$.
The respective connected squares are then obtained by ``opening the fan''.
The monoid unit laws give us the (generally) non-degenerate faces of a connected square,
while the \refer[dioid absorption law]{dioid absorption laws} give us the (necessarily) degenerate ones.
Monoid associativity says that multiply-connected higher-dimensional cubes are well-defined.

Because $[0]$ is terminal we also have the \define[connection-degeneracy law] {connection-degeneracy laws},
$\varepsilon \circ \mu_i = \varepsilon \otimes \varepsilon$,
which say that a connected square arising from a degenerate interval is a doubly-degenerate square.
In the presence of exchange, we further assume that the $\mu_i$ are commutative.
In the cubical semantics, this implies that connected cubes are invariant under reflection across their connected dimensions.

When adding contraction to the signature,
we must give a law for rewriting a diagonal following a connection such that the structural maps come first.
This is done by the \define[connection-diagonal law] {connection-diagonal laws},
$\delta \circ \mu_i = (\mu_i \otimes \mu_i) \circ (\id{} \otimes \tau \otimes \id{}) \circ (\delta \otimes \delta)$,
which we represent string-diagrammatically as:
\[
\begin{tikzpicture}[string diagram , baseline=(current bounding box.center)]
	\coordinate (tl) at (-0.6 , 0.6) ;
	\coordinate (br) at (0.6 , -0.6) ;
	\path [draw , clip] (tl) rectangle (br) ;
	\coordinate (mu) at ({$ (tl) ! 1/2 ! (br) $} |- {$ (tl) ! 1/4 ! (br) $}) ;
	\coordinate (del) at ({$ (tl) ! 1/2 ! (br) $} |- {$ (tl) ! 3/4 ! (br) $}) ;
	\coordinate (in1) at ({$ (tl) ! 1/5 ! (br) $} |- tl) ;
	\coordinate (in2) at ({$ (tl) ! 4/5 ! (br) $} |- tl) ;
	\coordinate (out1) at ({$ (tl) ! 1/5 ! (br) $} |- br) ;
	\coordinate (out2) at ({$ (tl) ! 4/5 ! (br) $} |- br) ;
	\draw (in1) to [out = south , in = west] (mu) ;
	\draw (in2) to [out = south , in = east] (mu) ;
	\draw (mu) to [out = south , in = north] (del) ;
	\draw (del) to [out = west , in = north] (out1) ;
	\draw (del) to [out = east , in = north] (out2) ;
	\node [2-cell] at (mu) {\mu_i} ;
	\node [2-cell] at (del) {\delta} ;
\end{tikzpicture}
\quad = \quad
\begin{tikzpicture}[string diagram , baseline=(current bounding box.center)]
	\coordinate (tl) at (-0.8 , 0.6) ;
	\coordinate (br) at (0.8 , -0.6) ;
	\path [draw , clip] (tl) rectangle (br) ;
	\coordinate (del1) at ({$ (tl) ! 1/4 ! (br) $} |- {$ (tl) ! 1/4 ! (br) $}) ;
	\coordinate (del2) at ({$ (tl) ! 3/4 ! (br) $} |- {$ (tl) ! 1/4 ! (br) $}) ;
	\coordinate (mu1) at ({$ (tl) ! 1/4 ! (br) $} |- {$ (tl) ! 3/4 ! (br) $}) ;
	\coordinate (mu2) at ({$ (tl) ! 3/4 ! (br) $} |- {$ (tl) ! 3/4 ! (br) $}) ;
	\coordinate (in1) at (del1 |- tl) ;
	\coordinate (in2) at (del2 |- tl) ;
	\coordinate (out1) at (mu1 |- br) ;
	\coordinate (out2) at (mu2 |- br) ;
	\draw (in1) to [out = south , in = north] (del1) ;
	\draw (in2) to [out = south , in = north] (del2) ;
	\draw (del1) to [out = west , in = west , looseness = 1.6] (mu1) ;
	\draw (del1) to [out = east , in = west] (mu2) ;
	\draw (del2) to [out = west , in = east] (mu1) ;
	\draw (del2) to [out = east , in = east , looseness = 1.6] (mu2) ;
	\draw (mu1) to [out = south , in = north] (out1) ;
	\draw (mu2) to [out = south , in = north] (out2) ;
	\node [2-cell] at (del1) {\delta} ;
	\node [2-cell] at (del2) {\delta} ;
	\node [2-cell] at (mu1) {\mu_i} ;
	\node [2-cell] at (mu2) {\mu_i} ;
\end{tikzpicture}
\]
Algebraically, this says that two copies of a conjunction consists of a pair of conjunctions, each on copies of the respective terms;
cubically, it gives the connected square on the diagonal interval of another square
in terms of a product of diagonals in the $4$-cube resulting from connecting each dimension separately.
The whole structure, then, is a pair of bicommutative bimonoids, $(\delta , \varepsilon , \mu_i , \eta_i)$
related by the \refer[dioid absorption law]{dioid absorption laws}.
We refer to these as \define[linked bimonoid]{linked bimonoids}.

Additionally, we may impose the laws of bounded, modular, or distributive lattices,
each of which implies its predecessors, and all of which imply the \define[diagonal-connection law] {diagonal-connection laws},
$\mu_i \circ \delta = \id{[1]}$, known in the literature (rather generically) as ``special'' laws.
These correspond to the ${\vee}$- and ${\wedge}$-idempotence laws of Tab.~\ref{tab:axioms}.

\subsubsection{The full signature} \label{sec:full_signature}

When both reversals and connections are present, the \define[de Morgan law]{de~Morgan laws},
$\rho \circ \mu_i = \mu_j \circ (\rho \otimes \rho)$ permute reversals before connections:
\[
\begin{tikzpicture}[string diagram , baseline=(current bounding box.center)]
	\coordinate (tl) at (-0.75 , 0.75) ;
	\coordinate (br) at (0.75 , -0.75) ;
	\path [draw , clip] (tl) rectangle (br) ;
	\coordinate (mu) at ({$ (tl) ! 1/2 ! (br) $} |- {$ (tl) ! 3/10 ! (br) $}) ;
	\coordinate (in1) at ({$ (tl) ! 1/5 ! (br) $} |- tl) ;
	\coordinate (in2) at ({$ (tl) ! 4/5 ! (br) $} |- tl) ;
	\coordinate (out) at (mu |- br) ;
	\draw (in1) to [out = south , in = west] (mu) ;
	\draw (in2) to [out = south , in = east] (mu) ;
	\draw (mu) to [out = south , in = north] coordinate [pos = 3/5] (rho) (out) ;
	\node [2-cell] at (mu) {\mu_i} ;
	\node [diamond] at (rho) {\rho} ;
\end{tikzpicture}
\quad = \quad
\begin{tikzpicture}[string diagram , baseline=(current bounding box.center)]
	\coordinate (tl) at (-1.0 , 0.75) ;
	\coordinate (br) at (1.0 , -0.75) ;
	\path [draw , clip] (tl) rectangle (br) ;
	\coordinate (mu) at ({$ (tl) ! 1/2 ! (br) $} |- {$ (tl) ! 2/3 ! (br) $}) ;
	\coordinate (in1) at ({$ (tl) ! 1/5 ! (br) $} |- tl) ;
	\coordinate (in2) at ({$ (tl) ! 4/5 ! (br) $} |- tl) ;
	\coordinate (out) at (mu |- br) ;
	\draw (in1) to [out = south , in = west] coordinate [pos = 1/3] (rho1) (mu) ;
	\draw (in2) to [out = south , in = east] coordinate [pos = 1/3] (rho2) (mu) ;
	\draw (mu) to [out = south , in = north] (out) ;
	\node [2-cell] at (mu) {\mu_j} ;
	\node [diamond] at (rho1) {\rho} ;
	\node [diamond] at (rho2) {\rho} ;
\end{tikzpicture}
\]
These laws imply each other as well as their nullary versions, the \refer[face-reversal law]{face-reversal laws}.

Using the algebraic characterization of order in a lattice, $x \wedge y = x  \Longleftrightarrow  x \le y  \Longleftrightarrow  x \vee y = y$,
we can express the \define[kleene law]{Kleene law} as
$\mu_i \circ (\mu_i \otimes \mu_j) \circ (\id{} \otimes \rho \otimes \id{} \otimes \rho) \circ (\delta \otimes \delta)
= \mu_i \circ (\id{} \otimes \rho) \circ (\delta \otimes \varepsilon)$:
\[
\begin{tikzpicture}[string diagram , baseline=(current bounding box.center)]
	\coordinate (tl) at (-1.25 , 1.0) ;
	\coordinate (br) at (1.25 , -1.0) ;
	\path [draw , clip] (tl) rectangle (br) ;
	\coordinate (delta1) at ({$ (tl) ! 1/4 ! (br) $} |- {$ (tl) ! 1/6 ! (br) $}) ;
	\coordinate (delta2) at ({$ (tl) ! 3/4 ! (br) $} |- {$ (tl) ! 1/6 ! (br) $}) ;
	\coordinate (mu1) at ({$ (tl) ! 1/4 ! (br) $} |- {$ (tl) ! 4/6 ! (br) $}) ;
	\coordinate (mu2) at ({$ (tl) ! 3/4 ! (br) $} |- {$ (tl) ! 4/6 ! (br) $}) ;
	\coordinate (mu3) at ({$ (tl) ! 1/2 ! (br) $} |- {$ (tl) ! 5/6 ! (br) $}) ;
	\coordinate (in1) at (delta1 |- tl) ;
	\coordinate (in2) at (delta2 |- tl) ;
	\coordinate (out) at (mu3 |- br) ;
	\draw (in1) to [out = south , in = north] (delta1) ;
	\draw (in2) to [out = south , in = north] (delta2) ;
	\draw (delta1) to [out = west , in = west] (mu1) ;
	\draw (delta1) to [out = east , in = east] coordinate [pos = 1/2] (rho1) (mu1) ;
	\draw (delta2) to [out = west , in = west] (mu2) ;
	\draw (delta2) to [out = east , in = east] coordinate [pos = 1/2] (rho2) (mu2) ;
	\draw (mu1) to [out = south , in = west] (mu3) ;
	\draw (mu2) to [out = south , in = east] (mu3) ;
	\draw (mu3) to [out = south , in = north] (out) ;
	\node [2-cell] at (delta1) {\delta} ;
	\node [2-cell] at (delta2) {\delta} ;
	\node [diamond] at (rho1) {\rho} ;
	\node [diamond] at (rho2) {\rho} ;
	\node [2-cell] at (mu1) {\mu_i} ;
	\node [2-cell] at (mu2) {\mu_j} ;
	\node [2-cell] at (mu3) {\mu_i} ;
\end{tikzpicture}
\quad = \quad
\begin{tikzpicture}[string diagram , baseline=(current bounding box.center)]
	\coordinate (tl) at (-1.0 , 1.0) ;
	\coordinate (br) at (1.0 , -1.0) ;
	\path [draw , clip] (tl) rectangle (br) ;
	\coordinate (delta) at ({$ (tl) ! 1/4 ! (br) $} |- {$ (tl) ! 1/4 ! (br) $}) ;
	\coordinate (epsilon) at ({$ (tl) ! 3/4 ! (br) $} |- {$ (tl) ! 1/4 ! (br) $}) ;
	\coordinate (mu) at ({$ (tl) ! 1/4 ! (br) $} |- {$ (tl) ! 3/4 ! (br) $}) ;
	\coordinate (in1) at (delta |- tl) ;
	\coordinate (in2) at (epsilon |- tl) ;
	\coordinate (out) at (mu |- br) ;
	\draw (in1) to [out = south , in = north] (delta) ;
	\draw (in2) to [out = south , in = north] (epsilon) ;
	\draw (delta) to [out = west , in = west] (mu) ;
	\draw (delta) to [out = east , in = east] coordinate [pos = 1/2] (rho) (mu) ;
	\draw (mu) to [out = south , in = north] (out) ;
	\node [2-cell] at (delta) {\delta} ;
	\node [2-cell] at (epsilon) {\varepsilon} ;
	\node [2-cell] at (mu) {\mu_i} ;
	\node [diamond] at (rho) {\rho} ;
\end{tikzpicture}
\]

Finally, we arrive at the structure of a boolean algebra by assuming the \define[hopf law]{Hopf laws},
$\mu_i \circ (\id{} \otimes \rho) \circ \delta = \eta_j \circ \varepsilon$:
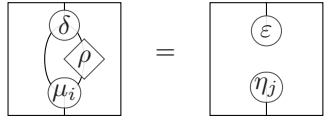
\[
\begin{tikzpicture}[string diagram , baseline=(current bounding box.center)]
	\coordinate (tl) at (-0.75 , 0.75) ;
	\coordinate (br) at (0.75 , -0.75) ;
	\path [draw , clip] (tl) rectangle (br) ;
	\coordinate (delta) at ({$ (tl) ! 1/2 ! (br) $} |- {$ (tl) ! 1/5 ! (br) $}) ;
	\coordinate (mu) at ({$ (tl) ! 1/2 ! (br) $} |- {$ (tl) ! 4/5 ! (br) $}) ;
	\coordinate (in) at (delta |- tl) ;
	\coordinate (out) at (mu |- br) ;
	\draw (in) to [out = south , in = north] (delta) ;
	\draw (delta) to [out = west , in = west] (mu) ;
	\draw (delta) to [out = east , in = east] coordinate [pos = 1/2] (rho) (mu) ;
	\draw (mu) to [out = south , in = north] (out) ;
	\node [2-cell] at (delta) {\delta} ;
	\node [2-cell] at (mu) {\mu_i} ;
	\node [diamond] at (rho) {\rho} ;
\end{tikzpicture}
\quad = \quad
\begin{tikzpicture}[string diagram , baseline=(current bounding box.center)]
	\coordinate (tl) at (-0.75 , 0.75) ;
	\coordinate (br) at (0.75 , -0.75) ;
	\path [draw , clip] (tl) rectangle (br) ;
	\coordinate (epsilon) at ({$ (tl) ! 1/2 ! (br) $} |- {$ (tl) ! 1/4 ! (br) $}) ;
	\coordinate (eta) at ({$ (tl) ! 1/2 ! (br) $} |- {$ (tl) ! 3/4 ! (br) $}) ;
	\coordinate (in) at (eta |- tl) ;
	\coordinate (out) at (eta |- br) ;
	\draw (in) to [out = south , in = north] (epsilon) ;
	\draw (eta) to [out = south , in = north] (out) ;
	\node [2-cell] at (epsilon) {\varepsilon} ;
	\node [2-cell] at (eta) {\eta_j} ;
\end{tikzpicture}
\]
Cubically, these say that the anti-diagonal of a connected square is degenerate.

\section{Test categories}
\label{sec:test-categories}

As mentioned in the introduction, a test category is a small category
$A$ such that the presheaf category $\widehat A$ can model the
homotopy category after a canonical localization. Here we recall the
precise definitions (cf.~\textcite{Maltsiniotis2005}).

The inclusion functor $i : \Delta \to \Cat$ induces via~\eqref{eq:adjunction} an adjunction $i_! : \widehat\Delta \leftrightarrows \Cat : N$ where the right adjoint $N$ is the \emph{nerve} functor.
This allows us to transfer the homotopy theory of simplicial sets to the setting of (strict) small categories, in that we define a functor $f : A \to B$ to be a \emph{weak equivalence} if $N(f)$ is a weak equivalence of simplicial sets.
A small category $A$ is called \emph{aspheric} (or \emph{weakly contractible}) if the canonical functor $A \to 1$ is a weak equivalence.
Since any natural transformation of functors induces a homotopy, it follows that any category with a natural transformation between the identity functor and a constant endofunctor is contractible, and hence aspheric.
In particular, a category with an initial or terminal object is aspheric.
It follows from Quillen's Theorem~A~\parencite{Quillen1973} that a functor $f : A \to B$ is a weak equivalence, if it is \emph{aspheric},
meaning that all the slice categories $A_{/b}$ are aspheric, for $b \in B$.%
\footnote{We follow Grothendieck and let the slice $A_{/b}$ denote the comma category $(f \downarrow b)$ of the functors $f:A\to B$ and $b:1\to B$. Similarly, the coslice ${}_{b\setminus{}}A$ denotes $(b \downarrow f)$.}
By duality, so is any \emph{coaspheric} functor $f : A \to B$, one for which the coslice categories ${}_{b\setminus{}}A$, for $b \in B$, are all aspheric.
Finally, we say that a presheaf $X$ in $\widehat A$ is \emph{aspheric}, if $A_{/X}$ is (note that $A_{/X}$ is the category of elements of $X$).

For any small category $A$, we can use the adjunction induced by the functor $i_A : A \to \Cat$, $i_A(a) = A_{/a}$ (as mentioned in subsection~\ref{sec:testcategories})
to define the class of weak equivalences in $\widehat A$, $\mathcal W_A$; namely,
$f : X \to Y$ in $\widehat A$ is in $\mathcal W_A$ if $i_A(f) : A_{/X} \to A_{/Y}$ is a weak equivalence of categories.
Following Grothendieck, we make the following definitions:
\begin{itemize}
\item $A$ is a \emph{weak test category} if the induced functor $\overline{i_A} : (\mathcal W_A)\inv\widehat A \to \Hot$ is an equivalence of categories (the inverse is then $\overline{i_A^*}$).
\item $A$ is a \emph{local test category} if all the slices $A_{/a}$ are weak test categories.
\item $A$ is a \emph{test category} if $A$ is both a weak and a local test category.
\item $A$ is a \emph{strict test category} if $A$ is a test category and the functor $\widehat A \to \Hot$ preserves finite products.
\end{itemize}

The goal of the rest of this section is to prove that all the cube categories \emph{with weakening} are test categories, and to establish exactly which ones are strict test categories. First we introduce Grothendieck interval objects. These allow us to show that any cartesian cube category is a strict test category (Cor.~\ref{cor:cart-strict}) as well as to show that all the cube categories are test categories (Cor.~\ref{cor:cube-test}). Then we adapt an argument of \textcite{Maltsiniotis2009} to show that any cube category with a connection is a strict test category (Thm.~\ref{thm:conn-strict}). Finally, we adapt another argument of his to show that the four remaining cube categories fail to be strict test categories (Thm.~\ref{thm:non-strict}). The end result is summarized in Tab.~\ref{tab:cubicaltest}.

\begin{table}[t]
  \centering
  \setlength{\tabcolsep}{12pt}
  \begin{tabular}{ccccccc} \toprule
 a$\smallsetminus$b   & $\zstr$ & ${'}$ & ${\vee}$ & ${\wedge}$
    & ${\vee}{\wedge}$ & ${\vee}{\wedge}'$ \\ \midrule
  w &  t  & t  & st & st & st & st \\
 we &  t  & t  & st & st & st & st \\
wec &  st & st & st & st & st & \makebox[0pt][c]{st/st/st} \\ \bottomrule
  \end{tabular}
  \vspace{6pt}
  
  \caption{Which canonical cube categories $\C_{(a,b)}$ are test (t) or even strict test (st) categories. The bottom-right corner refers to the cube categories corresponding to de~Morgan, Kleene and boolean algebras.}
  \label{tab:cubicaltest}
\end{table}

In order to study test categories, \textcite{Grothendieck1983} introduced the notion of an
\emph{interval} (\emph{segment} in the terminology of \textcite{Maltsiniotis2005}) in a presheaf category $\widehat A$. This is an object $I$ equipped with two global elements $d^0,d^1$. As such, it
is just a structure for the initial cubical language, $L_{(\zstr,\zstr)}$ in
the cartesian monoidal category $\widehat A$, and $\square^1$ is thus canonically a Grothendieck
interval in all our categories of cubical sets, $\widehat\C$.

An interval is \emph{separated} if the equalizer of $d^0$ and $d^1$ is
the initial presheaf, $\varnothing$. For cubical sets, this is the
case if and only if $0=1$ is not derivable in the base category, in any variable context.


The following theorem is due to \textcite[44(c)]{Grothendieck1983}, cf.~Thm.~2.6 of \textcite{Maltsiniotis2009}.
We say that $A$ is \emph{totally aspheric} if $A$ is non-empty and all the products $\yo(a) \times \yo(b)$ are aspheric, where $\yo : A \to \widehat A$ is the Yoneda embedding.

\begin{theorem}[Grothendieck]
  If $A$ is a small category that is totally aspheric and has a
  separated aspheric interval, then $A$ is a strict test category.
\end{theorem}
\begin{corollary}\label{cor:cart-strict}
  If $\C$ is any canonical cube category over the full set of structural rules, $\C_{(\cst{wec},b)}$, or the cartesian cube category of de~Morgan algebras or that of boolean algebras, then $\C$ is a strict test category.
\end{corollary}
\begin{proof}
  Since $\C$ has finite products it is totally aspheric, as the Yoneda embedding preserves finite products. The Grothendieck interval corresponding to the $1$-cube is representable and hence aspheric. This interval is separated as $0=1$ is not derivable in any context.\qed
\end{proof}

The following theorem is from \textcite[44(d), Prop.~on p.\ 86]{Grothendieck1983}:
\begin{theorem}[Grothendieck]
  If $A$ is a small aspheric category with a separated aspheric
  interval $(I,d^0,d^1)$ in $\widehat A$, and
  $i : A \to \Cat$ is a functor such that for any $a$ in $A$, $i(a)$
  has a final object, and there is a map of intervals
  $i_!(I) \to \bbtwo$ in $\Cat$, then $A$ is a test category.
\end{theorem}
In this case, $i$ is in fact a \emph{weak test functor}, meaning that $i^* : \Cat \to \widehat A$ induces an equivalence $\Hot \to (\mathcal W_A)\inv\widehat A$.
\begin{corollary}\label{cor:cube-test}
  Any canonical cube category $\C$ is a test category.
\end{corollary}
\begin{proof}
  The conditions of the theorem hold trivially for any cube category
  without reversal, taking $i$ to be the canonical functor sending $I$
  to $\bbtwo$. If we have reversals, then we can define $i$ by sending
  $I$ to the category with three objects, $\{0\},\{1\},\{0,1\}$ and
  arrows $\{0\},\{1\}\to\{0,1\}$.  Since this is a Kleene algebra,
  this functor is well-defined in all cases.\qed
\end{proof}
The cube categories $\C_{\mathrm{dM}}$ and $\C_{\mathrm{BA}}$ of de~Morgan and boolean algebras are test categories by Cor.~\ref{cor:cart-strict}.

We now turn to the question of which non-cartesian cube categories are strict test categories.
The following theorem was proved by \textcite[Prop.~3.3]{Maltsiniotis2009} for the case of $\C_{(\cst w,{\vee})}$. The same proof works more generally, so we obtain:

\begin{theorem}\label{thm:conn-strict}
  Any canonical cube category $\C_{(a,b)}$ where $b$ includes one of the connections $\vee$, $\wedge$ is a strict test category.
\end{theorem}

That leaves four cases: $(\cst w,\zstr)$, $(\cst w,{'})$, $(\cst{we},\zstr)$, $(\cst{we},{'})$.
The first of these, the ``classical'' cube category, is not a strict test
category by the argument of \textcite[Sec.~5]{Maltsiniotis2009}.

Note the unique factorizations we have for these categories:
every morphism $f:[m]\to[n]$
factors as degeneracies, followed
by (possibly) symmetries, followed by (possibly) reversals, followed
by face maps. The isomorphisms are exactly the compositions of
reversals and symmetries (if any).

Next, we use variations of the argument of Maltsiniotis (\emph{loc.~cit.}) to show that none of these four sites are
strict test categories by analyzing the homotopy type of the slice
category $\C_{/\square^1\times\square^1}$ in each case.
\begin{theorem}\label{thm:non-strict}
  The canonical cube categories $\C_{(\cst w,\zstr)}$, $\C_{(\cst w,{'})}$, $\C_{(\cst{we},\zstr)}$ and $\C_{(\cst{we},{'})}$ are not strict test categories.
\end{theorem}
\begin{proof}
  Let $\C$ be one of these categories.
  We shall find a full subcategory $A$ of $\C_{/\square^1\times\square^1}$ that is not aspheric, and such that the inclusion is a weak equivalence. Hence $\C_{/\square^1\times\square^1}$ is not aspheric, and $\C$ cannot be a strict test category.

  An object of the slice category $\C_{/\square^1\times\square^1}$ is given by a dimension and two terms corresponding to that dimension.
  We can thus represent it by a variable context $x_1\cdots x_n$ (for an $n$-cube $[n]$) and two terms $s,t$ in that context.
  
  For the classical cube category case $(\cst w,\zstr)$, we let $A$ contain the following objects, cf.~\textcite[Prop.~5.2]{Maltsiniotis2009}:
  \begin{itemize}
  \item $4$ zero-dimensional objects $(\zstr,(i,j))$, $i,j\in\{0,1\}$.
  \item $5$ one-dimensional objects; $4$ corresponding to the sides of a square, $(x,(i,x))$ and $(x,(x,i))$ with $i\in\{0,1\}$, and $1$ corresponding to its diagonal, $(x,(x,x))$.
  \item $2$ two-dimensional objects: $(xy, (x,y))$ and $(xy, (y,x))$ (let us call them, respectively, the northern and southern hemispheres).
  \end{itemize}

  In the presence of the exchange rule, we do not need both of the two-dimensional objects (so we retain, say, the northern hemisphere), and in the presence of reversal, we need additionally the anti-diagonal, $(x,(x,x'))$.

  Note that in each case $A$ is a partially ordered set (there is at most one morphism between any two objects), and we illustrate the incidence relations between the objects in Fig.~\ref{fig:posets}.
  The figure also illustrates how to construct functors $F : A \to \Top$, which are cofibrant with respect to the Reedy model structure on this functor category, where $A$ itself is considered a directed Reedy category relative to the obvious dimension assignment (points of dimension zero, lines of dimension one, and the hemispheres of dimension two).
  We conclude that the homotopy colimit of $F$ (in $\Top$) is weakly equivalent to the ordinary colimit of $F$, which is seen to be equivalent to $S^2 \vee S^1$, $S^2\vee S^1 \vee S^1$, $S^1$ and $S^1 \vee S^1$, respectively for the cases $(\cst w,\zstr)$, $(\cst w,{'})$, $(\cst{we},\zstr)$ and $(\cst{we},{'})$.
  Since $F$ takes values in contractible spaces, this homotopy colimit represents in each case the homotopy type of the nerve of $A$. It follows that $A$ is not aspheric.

\begin{figure}[t]
\centering
\begin{tikzpicture}
 \draw (-4,5) arc (180:360:2 and 0.5)
   node[pos=0.25,anchor=north] (a00) {$(0,0)$}
   node[pos=0.75,anchor=north] (a10) {$(1,0)$};
 \fill (a00.north) circle(2pt) (a10.north) circle(2pt);
 \draw [dashed] (-4,5) arc (180:0:2 and 0.5)
   node[pos=0.25,anchor=south] (a01) {$(0,1)$}
   node[pos=0.75,anchor=south] (a11) {$(1,1)$};
 \fill (a01.south) circle(2pt) (a11.south) circle(2pt);
 \draw (a00.north) -- (a11.south);
 
 \draw (2,5) arc (180:360:2 and 0.5)
   node[pos=0.25,anchor=north] (b00) {$(0,0)$}
   node[pos=0.75,anchor=north] (b10) {$(1,0)$};
 \fill (b00.north) circle(2pt) (b10.north) circle(2pt);
 \draw [dashed] (2,5) arc (180:0:2 and 0.5)
   node[pos=0.25,anchor=south] (b01) {$(0,1)$}
   node[pos=0.75,anchor=south] (b11) {$(1,1)$};
 \fill (b01.south) circle(2pt) (b11.south) circle(2pt);
 \pgfmathanglebetweenpoints{\pgfpointanchor{b01}{south}}{\pgfpointanchor{b10}{north}}
 \edef\endangle{\pgfmathresult}
 \pgfmathsetmacro\startangle{-180+\endangle}
 \draw (b00.north) -- (b11.south);
 \draw (b01.south) -- ($ (4,5)+(\startangle:.2) $)
    arc[start angle=\startangle,delta angle=-180,radius=0.2]
    ($ (4,5)+(\endangle:0.2) $) -- (b10.north);

 \draw (-4,0) arc (180:360:2 and 0.5)
   node[pos=0.25,anchor=north] (c00) {$(0,0)$}
   node[pos=0.75,anchor=north] (c10) {$(1,0)$};
 \fill (c00.north) circle(2pt) (c10.north) circle(2pt);
 \draw [dashed] (-4,0) arc (180:0:2 and 0.5)
   node[pos=0.25,anchor=south] (c01) {$(0,1)$}
   node[pos=0.75,anchor=south] (c11) {$(1,1)$};
 \fill (c01.south) circle(2pt) (c11.south) circle(2pt);
 \draw (c00.north) -- (c11.south);
 
 \draw (2,0) arc (180:360:2 and 0.5)
   node[pos=0.25,anchor=north] (d00) {$(0,0)$}
   node[pos=0.75,anchor=north] (d10) {$(1,0)$};
 \fill (d00.north) circle(2pt) (d10.north) circle(2pt);
 \draw [dashed] (2,0) arc (180:0:2 and 0.5)
   node[pos=0.25,anchor=south] (d01) {$(0,1)$}
   node[pos=0.75,anchor=south] (d11) {$(1,1)$};
 \fill (d01.south) circle(2pt) (d11.south) circle(2pt);
 \draw (d00.north) -- (d11.south);
 \draw (d01.south) -- ($ (4,0)+(\startangle:.2) $)
    arc[start angle=\startangle,delta angle=-180,radius=0.2]
    ($ (4,0)+(\endangle:0.2) $) -- (d10.north);

 \draw (0,5) arc (360:0:2 and 2);
 \draw (6,5) arc (360:0:2 and 2);
 \draw (-4,0) arc (180:0:2 and 2);
 \draw (2,0) arc (180:0:2 and 2);

 \draw (-4.75,5) node {$A_{(\cst w,\zstr)} =$}
       (-4.75,0.5) node {$A_{(\cst{we},\zstr)} =$}
       (1.25,5) node {$A_{(\cst w,{'})} =$}
       (1.25,0.5) node {$A_{(\cst{we},{'})}=$};
    \end{tikzpicture}
    \caption{The partially ordered sets $A = A_{(a,b)}$ in their topological realizations.}
    \label{fig:posets}
  \end{figure}
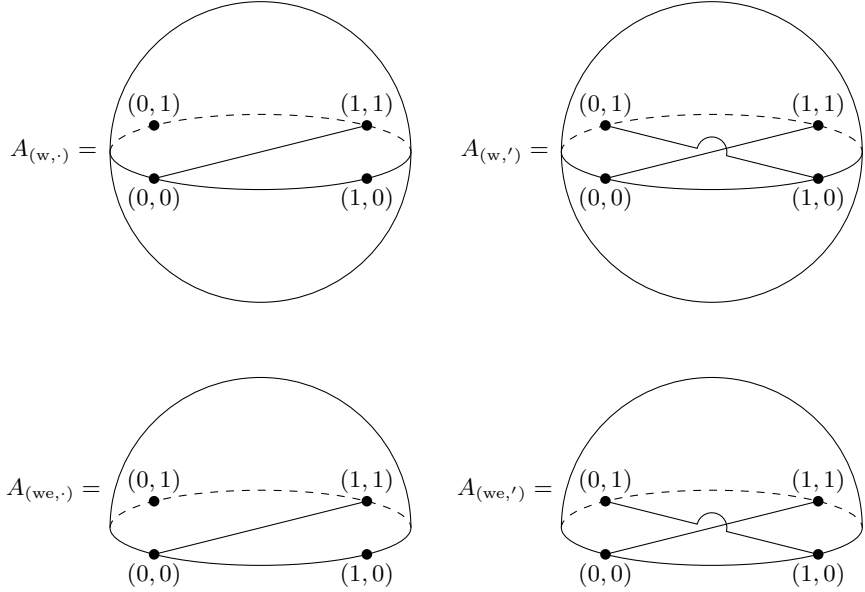

  It remains to see that the inclusion $A \hookrightarrow \C_{/\square^1\times\square^1}$ is in each case a weak equivalence.
  For this we use the dual of Quillen's Theorem~A, and show that for each object $(x_1\cdots x_n,(s,t))$ (written $(s,t)$ for short) of $\C_{/\square^1\times\square^1}$, the coslice category $B_{(s,t)}={}_{(s,t)\setminus}A$ has an initial object, and is hence aspheric.

  As in Maltsiniotis (\emph{loc.~cit.}), we make case distinctions on $(s,t)$.
  Note that for the languages under consideration, a term can have at most one free variable.
  \begin{itemize}
  \item $(s,t)=(x('),y('))$ for distinct variables $x,y$. The initial object is a hemisphere (the northern one in the presence of exchange, the southern one without exchange and in case $x,y$ appear in reversed order in the variable context).
  \item $(s,t)=(x,x)$ or $(x',x')$. The initial object is then the diagonal in $A$.
  \item $(s,t)=(x,x')$ or $(x',x)$. The initial object is here the anti-diagonal.
  \item $(s,t)=(i,x('))$ or $(x('),i)$. The initial object is the corresponding side of the square.
  \item $(s,t)=(i,j)$. The initial object is then the corresponding vertex of the square.\qed
  \end{itemize}
\end{proof}
We remark that this proof method fails in the presence of a connection, where terms can now refer to multiple variables.
And in the presence of contraction, the diagonal $(x,x)$ is incident on the northern hemisphere, which is then a maximal element in $A$, and hence $A$ is contractible.
These observations explain why it is only those four cases that can fail to give strict test categories.

\section{Conclusion}
\label{sec:conclusion}

We have espoused a systematic algebraic framework for describing notions of cubical sets, and we have shown that all reasonable cubical sites are test categories.
We have shown that a cubical site is a strict test category precisely when it is cartesian monoidal or includes one of the connections.

To improve our understanding of the homotopy theory of each cubical site, we need to investigate the induced Quillen equivalence between the corresponding category of cubical sets $\widehat\C$ with the Cisinski model structure and the category of simplicial sets $\widehat\Delta$ with the Kan model structure.
For instance, we can ask whether the fibrations are those satisfying the cubical Kan filling conditions.
Furthermore, to model type theory, one should probably require that the model structures lift to algebraic model structures. We leave all this for future work.

\section{Acknowledgements}
\label{sec:acknowledgements}

We wish to thank the members of the HoTT group at Carnegie Mellon University for many fruitful discussions, in particular Steve Awodey who has encouraged the study of cartesian cube categories since 2013 and who has been supportive of our work, as well as Bob Harper who has also been very supportive.
Additionally, we deeply appreciate the influence of Bas Spitters who inspired us with a seminar presentation of a different approach to showing that $\C_{(\cst{wec},\zstr)}$ is a strict test category.

The authors gratefully acknowledge the support of the Air Force Office of Scientific Research through MURI grant FA9550-15-1-0053.
Any opinions, findings and conclusions or recommendations expressed in this material are those of the authors
and do not necessarily reflect the views of the AFOSR.

\ifarxiv
\bibliographystyle{hplainnat}
\bibliography{cubical}
\else
\printbibliography
\fi
\end{document}
